\documentclass[12pt]{amsart}
\usepackage{graphicx}
\newtheorem{theorem}{Theorem}
\newtheorem{definition}[theorem]{Definition}
\begin{document}

\title[Uniqueness sets]{Uniqueness sets for Fourier series}
% Class\thanks{Version 2.0, 1999/11/15}}
\author[Vagharshakyan]{Ashot Vagharshakyan}
\address{Institute of Mathematics NAS}
\address{ of Armenia, Bagramian 24-b,}
\address{375019 Yerevan, Armenia}

\begin{abstract}
The paper discusses some uniqueness sets for Fourier series.
\end{abstract}
\maketitle

\section{Introduction}

In this paper the following problem is considered: to find 
conditions on a set $E\in [-\pi,\,\pi]$ such that, if a function $f(x),\,-\pi<x<\pi,$ belongs to some space and its fourier 
series converges to zero at each point of the set $E$, then $f(x)$
is identically zero. 

	The first nontrivial result, for trigonometric series, was 
proved by G. Cantor and W. Young, see \cite{Bary}, p. 191.
\begin{theorem} 
Let $c_k \to 0$ and for each point 
$x\in [-\pi,\,\pi]\setminus F$ 
we have
\[
\lim_{n\to \infty}\sum_{k=-n}^nc_ke^{ikx}=0,
\]
where $F$ be a countable set. Then
$c_k=0,\,\,\,k\in Z$.
\end{theorem}
D. Menshov, see \cite{Bary}, p. 806, construct a nonzero 
measure $\mu$ which has support of zero Lebesgue measure, 
and its Fourier coefficients tend to zero. The partial 
summs, of that Fourier series converges to zers out of 
$supp(\mu)$.  

\section{Auxliary definitions and results}

More information, about the following quantities, related with
Hausdorff's measures and capacities, one can find in \cite{Car}, 
pp. 13 - 46. For convenient of the reader, we give some definitions.
\begin{definition} 
Let $0\leq h(x),\,\,0\leq x\leq 1$ be a nonnegative, increasing 
function and $h(0)=0$. Let the subset 
$E\subset \{z;\quad |z|=1\}$ 
be cover by the family of arcs  
$\{S_k\}_{k=1}^{\infty}$, 
i.e.
\begin{equation*}
E\subseteq \bigcup_{k=1}^{\infty}S_k.
\end{equation*}
Then we put
\begin{equation*}
M_{h}(E)=\inf \left(\sum_{k=1}^{\infty}h(|S_k|)\right),
\end{equation*}
where 
$|S|$ 
is the length of the arc $S$ and the infimum is taken over all 
families of cover.
\end{definition}
\begin{definition}         
Let $0<\alpha<1$ and $E$ be baunded Borel set. The $C_{\alpha}(E)$- 
capactiy of the set $E$ is defined by formula
\begin{equation*}
C_{\alpha}(E)=\left(\inf_{\mu\prec E}\int_{E}\int_{E} 
\frac{d\mu(x)d\mu(y)}{|x-y|^{\alpha}}\right)^{-1},
\end{equation*}
where $\mu \prec E$ means, that $\mu$ is probality measure with support in $E$.
\end{definition}

For each $0<\alpha<1$ from Parseval's equality follows that there is a constant $M$ such that,
\begin{equation*}
\sum_{k=-\infty}^{\infty}|\hat{f}_{k}|^2|k|^{\alpha}\leq M\,\int_{-\pi}^{\pi}|f(x)|^2dx+M \,\int_{-\pi}^{\pi}\int_{\pi}^{\pi}
\frac{|f(x)-f(y)|^2}{|x-y|^{1+\alpha}}dxdy.
\end{equation*}
The following statement is a fragment of A. Zygmund's theorem, see \cite{K}, p.22. Let $g(-\pi)=g(\pi)$ and $g(x),\,-\pi\leq x\leq \pi]$ have a bounded variation. If
\begin{equation*}
|g(x)-g(y)|\leq M\,\cdot h(|x-y|).
\end{equation*}
Then, there is a constant $B$ such that the Fourier coefficients of the function $g(x)$ satisfy the inequalities
\begin{equation*}
\sum_{2^j\leq |k|<2^{j+1}}|\hat{g}_{k}|^2\leq
 B\,2^{-j}h\left(\frac{\pi}{3}2^{1-j}\right).
\end{equation*}
\begin{definition}
Let us denote by $\Lambda(n)$ the known Mangold's function, equals
\begin{equation*}
\Lambda\left(p^n\right)=\ln p,
\end{equation*}
where $p$ is prime number and $n$ is natural number. 
For other natural numbers $m$ 
\begin{equation*}
\Lambda\left(m\right)=0.
\end{equation*}
\end{definition}
It is known, that for an arbitrary 
natural number $n$ the equality
\begin{equation*}
\ln n=\sum_{d|n}\Lambda(d)
\end{equation*}
holds, where the sum is taken over all positive divisors of $n$. 

	In the following theorem A. Broman, see \cite{B}, p. 851, 
gived the characterization of close exeptional sets.
\begin{theorem} 
Let $0<\alpha<1$ and
\[
\sum^{\infty}_{n=-\infty}\frac{|c_n|^2}{|n|^{\alpha}+1}<\infty.
\]
Let $F$ be close set and
\[
\lim_{r\to 1-0}\sum_{k=-\infty}^{\infty}r^{|k|}c_ke^{ikx}=0,
\]
for arbitrary $x\in [-\pi,~\pi]\setminus F$. 
Then $c_k=0,\,\,\,k\in Z,$ if and only if 
\[
C_{1-\alpha}(F)=0.
\]
\end{theorem}

	A. Zygmund, see \cite{Z}, proved the following nontrivial result. 
\begin{theorem}
Let $\epsilon>0$ and $\epsilon_n >0,\,\,n=1,2,\dots$ be an 
arbitrary decreasing sequence, tending to zero. Let 
$|c_n|\leq \epsilon_n,\,\,n=1,2,\dots$. There is a subset 
$E\subseteq [-\pi,\,\,\pi]$ with measure, i. e. $m(E)>2\pi-\epsilon$, 
such that, if for each point $x\in [-\pi,\,\,\pi]\setminus E$ we have
\begin{equation*}
\lim_{n\to \infty}\sum_{k=-n}^nc_ke^{ikx}=0,
\end{equation*}
then $c_k=0,\,\,\,k\in Z,$.
\end{theorem}
The proof of the following theorem one can find in \cite{Vag2}.
\begin{theorem} 
Let $0\leq \alpha<1$ and 
\begin{equation*}
\int^{\pi}_{-\pi}|f(x)|dx+\int^{\pi}_{-\pi}\int^{\pi}_{-\pi}
\frac{|f(x)-f(y)|}{|x-y|^{1+\alpha}}dx < \infty
\end{equation*}
Let $E\subset [-\pi,\,\pi]$ be a subset such that for almost 
all points $x_0\in [-\pi,\,\pi]$ we have
\begin{equation*}
\sum_{n=1}^{\infty}2^{n(1-\alpha)}C_{1-\alpha}(E_n(x_0))=\infty,
\end{equation*}
where
\begin{equation*}
E_n(x_0)=\left\{x\in E;\,\,\frac{1}{2^{n+1}} 
\leq |x-y|<\frac{1}{2^{n}}\right\}.
\end{equation*}
If
\begin{equation*}
\lim_{n\to \infty}\sum_{k=-n}^n\hat{f}_ke^{ikx}=0,\,\,x\in E
\end{equation*}
then $f(x)=0,\,\,x\in [-\pi,\,\,\pi]$.
\end{theorem}
	In this paper we prove a new result of this type for 
other classes of functions.

\section{New uniqueness result}

The following result, in different form,
one can find in the paper, see \cite{Vag1}.
\begin{theorem}
Let $f(-\pi)=f(\pi)$ be differetable function.  
Then
\begin{equation*}
\frac{1}{\pi}\sum_{p\in P}\ln p 
\left(\sum_{n=1}^{\infty}\left[\frac{1}{p^n}
\sum_{k=1}^{p^n}f\left(\frac{2\pi k}{p^n}\right)
-\hat{f}_0\right]\right)=
\sum_{n\neq 0, n=-\infty}^{\infty}\hat{f}_{n}\ln |n|.
\end{equation*}
\end{theorem}
\begin{proof}
We have
\begin{equation*}
\sum_{n=1}^{\infty}\Lambda(n)\left(\frac{1}{n}
\sum_{k=1}^{n}
\frac{1-|z|^2}{\left|1-z\exp \{-\frac{2\pi i k}{n}\}\right|^2}-1\right)=
\end{equation*}
\begin{equation*}
=\sum_{n=1}^{\infty}\Lambda(n)\frac{1}{n}
\sum_{k=1}^{n}\left(\sum_{j=-\infty}^{\infty}r^{|j|}
\exp\left\{ixj-\frac{2\pi i k j}{n}\right\}-1\right)=
\end{equation*}
\begin{equation*}
=\sum_{j=1}^{\infty}r^j \left( \sum_{j/n}\Lambda(n)\right)\cos(jx)=
\sum_{j=1}^{\infty}r^j \cos(jx)\ln j.
\end{equation*}
where $z=re^{ix},\,\,0\leq r<1$.Multiplying by the function $f(x)$ and after integrating we get
\begin{equation*}
\sum_{n=1}^{\infty}\Lambda(n)\left(\frac{1}{n}\sum_{k=1}^{n}\frac{1}{2\pi}
\int_{-\pi}^{\pi}
\frac{1-r^2}{\left|1-r\exp\{ix-\frac{2\pi i k}{n}\}\right|^2}f(x)dx
-\hat{f}_{0}\right)=
\end{equation*}
\begin{equation*}
=\frac{1}{2}\sum_{j\neq 0,\,j=-\infty}^{\infty}r^{|j|}\hat{f}_{j}\ln |j|.
\end{equation*}
Passing to the limit if 
$r\to 1-0$ 
we get the required equality.
\end{proof}
{\bf Remark.} The getting result we can write in the form
\begin{equation*}
\frac{1}{\pi}\sum_{p\in P}\ln p 
\left(\sum_{n=1}^{\infty}\left[\frac{1}{p^n}
\sum_{k=1}^{p^n}\delta\left(x-\frac{2\pi k}{p^n}\right)
-1\right]\right)=\sum_{n\neq 0,\,n=-\infty}^{\infty}
e^{inx}\ln n.
\end{equation*}
This formula is a geralization of Poisson's well  known formula:
\begin{equation*}
\sum_{n=-\infty}^{\infty}\delta(x-2\pi n)=\frac{1}{2\pi}\sum_{n=-\infty}^{\infty}
e^{inx},\,-\infty<x<\infty.
\end{equation*}
\begin{theorem} 
Let $0<\alpha<1$ and a nonnegative function $x\leq h(x),\,\,0\leq x <1$ satisfy the condition 
\begin{equation*}
\int_0^1 \frac{h(x)}{x^{2-\alpha}}\ln^2\frac{e}{x}dx<\infty.
\end{equation*}
Let for the function $f(x)$ we have
\begin{equation*}
\int_{-\pi}^{\pi}|f(x)|^2dx+\int_{-\pi}^{\pi}\int_{\pi}^{\pi}
\frac{|f(x)-f(y)|^2}{|x-y|^{1+\alpha}}dxdy<\infty.
\end{equation*}
Let there is a subset $E\subset [-\pi,\,\pi]$ such that:

1. $M_{h}(E)>0$, 

2. if $x\in E$ then an arbitrary point  
\begin{equation*}
x+\frac{2\pi k}{p^n},
\end{equation*}
where 
$k\in Z,\,\, n\in N$ 
and 
$p$ 
is prime number, by $mod (2\pi)$, belongs to the set $E$. 

	If for eavery $x\in E$ we have
\begin{equation*}
\lim_{r\to 1-0}\sum_{k=-\infty}^{\infty}r^{|k|}\hat{f}_{k}e^{ikx}=0,
\end{equation*}
then the function $f(x)$ is identicaly zero.
\end{theorem}
\begin{proof}
By O. Frostman's theorem, see \cite{Car}, p. 14, there is a 
probability measure $d\mu$ such that $supp(d\mu)\subseteq E$,
\begin{equation*}
M_{h}(supp(\mu)) > 0.
\end{equation*}
and for each $0<\delta$ the inequality
\begin{equation*}
\int_{[x,\,x+\delta]}d\mu \leq A h(\delta)
\end{equation*}
hold. Let us assume at the points $0$ and $2\pi$ the function $\mu$ is 
continuous and $\mu(0)+1=\mu(2\pi)$. Denote 
\begin{equation*}
f(re^{ix})=\sum_{n=-\infty}^{\infty}r^{|n|}\hat{f}_ne^{inx}.
\end{equation*}
Then we have
\begin{equation*}
\frac{1}{\pi}\sum_{p\in P}\ln p 
\left(\sum_{n=1}^{\infty}\left[\frac{1}{p^n}
\sum_{k=1}^{p^n}\int_E f\left(r\exp\left\{\frac{2\pi i k}{p^n}+ix\right\}\right)d\mu(x)
-\hat{f}_0\right]\right)=
\end{equation*}
\begin{equation*}
=\sum_{n=2}^{\infty}r^n\left[\hat{f}_n\int_Ee^{inx}d\mu(x)+
\hat{f}_{-n}\int_Ee^{-inx}d\mu(x)\right]\ln n=
\end{equation*}
\begin{equation*}
=2\pi i\sum_{n=2}^{\infty}
\left(\hat{f}_n\hat{g}_{-n}-\hat{f}_{-n}\hat{g}_{-n}\right)r^nn\ln n.
\end{equation*}
where
\begin{equation*}
g(x)=\mu(x)-\frac{x}{2\pi}.
\end{equation*}
Since the function $f(e^{ix})$ vanish on the set $E$ so, we have
\begin{equation*}
\left|\frac{1}{\pi}\sum_{p\in P} \hat{f}_0\ln p \right|\leq 
2\sum_{n=2}^{\infty}(|\hat{f}_{-n}||\hat{g}_n|+|\hat{f}_{n}||\hat{g}_{-n}|)n\ln n.
\end{equation*}
Let us note
\begin{equation*}
\sum_{n=2}^{\infty}|\hat{f}_n||\hat{g}_{-n}|n\ln n\leq 
\left(\sum_{n=1}^{\infty}|\hat{f}_n|^2n^{\alpha} \right)^{\frac{1}{2}}
\left(\sum_{n=1}^{\infty}|\hat{g}_{-n}|^2n^{2-\alpha}\ln^2 n \right)^{\frac{1}{2}}\leq
\end{equation*}
\begin{equation*}
\leq\left(\sum_{n=1}^{\infty}|\hat{f}_n|^2n^{\alpha} \right)^{\frac{1}{2}}
\left(\sum_{j=1}^{\infty}j^22^{(2-\alpha)j}
\sum_{n=2^j}^{2^{j+1}-1}|\hat{g}_{-n}|^2\right)^{\frac{1}{2}}\leq
\end{equation*}
\begin{equation*}
\leq M \left(\sum_{j=1}^{\infty} j^22^{(1-\alpha)j}h(2^{-j})\right)^{\frac{1}{2}}<\infty.
\end{equation*}
The inequality 
\begin{equation*}
\left|\frac{1}{\pi}\sum_{p\in P}\hat{f}_0 \ln p \right|<\infty
\end{equation*}
valid only if 
\begin{equation*}
\hat{f}_0=0.
\end{equation*}
Considering the functions $e^{inx}f(x),\,\,\,n\in Z$ we prove that $\hat{f}(n)=0,\,\,n\in Z.$
\end{proof}

\end{document}